\documentclass[a4paper,10pt]{article}
\usepackage[T1]{fontenc}
\usepackage{graphicx,amsmath,amsfonts,amssymb,amsthm,psfrag,color 
}

\usepackage{bbm}
\usepackage{subfig}

\newcommand{\eps}{\varepsilon}

\newtheorem{thm}{Theorem}[section]

\newtheorem{prop}[thm]{Proposition}

\numberwithin{equation}{section}

\newtheorem{Algo}{Algorithm}

\begin{document}
\title{The first-passage time of the Brownian motion to a
curved boundary: an algorithmic approach\footnote{This work has been supported by the Agence National de la Recherche 
through the ANR Project MANDy ``Mathematical Analysis of Neuronal Dynamics'', ANR-09-BLAN-0008-01.}}
\author{Samuel Herrmann$^1$ and Etienne Tanr\'e$^2$}
\footnotetext[1]{Institut de Math\'ematiques de Bourgogne, 
IMB UMR5584, CNRS, Universit\'e de Bourgogne Franche-Comt\'e, F-21000 Dijon, France.
\texttt{Samuel.Herrmann@u-bourgogne.fr}}
\footnotetext[2]{Inria, Equipe Projet TOSCA, 2004 route des Lucioles, BP93,
06902 Sophia-Antipolis, France,  \texttt{Etienne.Tanre@inria.fr}}
 \date{\today}
\maketitle

\begin{abstract} Under some weak conditions, the first-passage time of the Brownian motion to a continuous curved 
boundary is an almost surely finite stopping time. Its probability density function (pdf) is explicitly known only in few particular cases. 
Several mathematical studies proposed to approximate the pdf in a quite general framework or even to simulate 
this hitting time using a discrete time approximation of the Brownian motion. The authors study a new algorithm 
which permits to simulate the first-passage time using an iterating procedure. The convergence rate presented in 
this paper suggests that the method is very efficient. 
\end{abstract}
\par\medskip

 \par\medskip

\section*{Introduction}
Modeling biological or physical systems often requires handling one-dimensional diffusion processes. The marginal probability distribution of such processes, at a fixed time, permits 
a quite precise description of the model. Nevertheless, in many applications, this information is insufficient and the description of the whole path becomes crucial. This is namely the case for a variety of problems related to neuronal sciences, financial derivatives with barriers, ruin probability of an insurance fund, optimal stopping problems,... In these frameworks, the main task is the description of the first passage time densities for time-dependent boundaries. Let us just mention some references in engineering reliability \cite{Ebrahimi-2005}, epidemiology \cite{Tuckwell-Wan-2000}, biology \cite{Ricciardi-al-1999}, mathematical finance \cite{Garrido-1989, Roberts-al-1997, Novikov-al-2003} and references concerning the framework of level-crossing problems \cite{Abrahams, Blake}.

For instance, let us focus our attention on a simple interpretation of neural transmission. When a neuron is stimulated by pressure, heat, light, or chemical information, its membrane voltage changes as time elapses and, as soon as it reaches a constant threshold, the depolarization phenomenon occurs and the voltage is reset to a resting potential. The family of integrate-and-fire spiking neuron models is based on this simple interpretation. The firing time therefore corresponds to the first-passage time of the membrane potential, represented by a stochastic mean-reverting process (usually the Ornstein-Uhlenbeck process) to the neural threshold (Giorno et al. \cite{Giorno-al-1988}, Lansky et al. \cite{Lansky-1995}, Wan and Tuckwell \cite{Wan-Tuckwell-1982}, for an introduction to noise in the nervous system see Part I Chapter 5 in \cite{Gerstner}, for the integrate-and-fire model see Chapter 10 in \cite{Ermentrout}).

Our main motivation is to emphasize an algorithmic approach in order to approximate the first-passage time of the Brownian motion to curved boundaries. The field of application of such an algorithm at a first glance may appear as quite restrictive since it concerns the Brownian motion but in fact a lot of families of diffusion processes are concerned. Indeed it is possible to express various stochastic paths as functions of Brownian paths in the spirit of Wang and P\"{o}tzelberger \cite{Wang-Potzelberger-2007}. 
Hence using simple time transformations, we are going to present an application of the results to the Ornstein-Uhlenbeck process (see Section~\ref{sec:examples}).

In order to describe approximations of the first-passage time of the Brownian motion, 
we assume that this stopping time is almost surely finite. 
In this way, we introduce particular conditions for this property to be satisfied. 
Let us consider a continuous function $\varphi:\mathbb{R}_+\to\mathbb{R}$ 
satisfying the following hypothesis:
\begin{equation}
\label{hypo}
\varphi(0)>0\quad\mbox{and}\quad \limsup_{t\to\infty}\frac{\varphi(t)}{\sqrt{2t\log\log t}} <1.\tag{H1}
\end{equation}
We then define the hitting time
\begin{equation}
\label{def:tau}
\tau_\varphi=\inf\{t>0:B_t=\varphi(t)\}
\end{equation}
where $(B_t,\,t\ge 0)$ stands for a standard one-dimensional Brownian motion. Under \eqref{hypo}, the a.s. finiteness of $\tau_\varphi$ is an obvious consequence of the law of the iterated logarithm (see e.g.\cite[Th.9.23 p.112]{karatzas}). It is quite difficult to obtain precise information about this stopping time in general situations. 

The study of the approximation of the hitting times for Brownian motion and general Gaussian Markov processes is an active area of research. Several alternatives for dealing with the characterization of hitting times exist. 

\subsubsection*{Approximation of the probability density function}

For particular cases, the probability density function of the Brownian passage time can be computed explicitly. Lerche \cite{Lerche} used the method of images in order to obtain explicit expressions of the p.d.f. $p$ defined by $p(t)\,dt=\mathbb{P}(\tau_\varphi\in dt)$. However only few cases are concerned by such a study.

Durbin \cite{Durbin85, Durbin92} proposed to approximate the first-passage distribution $p(t)$ of the Brownian motion as follows: $p$ can be represented by an expansion
\[
p(t)=\sum_{j=1}^k (-1)^{j-1}q_j(t)+(-1)^k r_k(t),\quad k\ge 1,
\]
where $q_j$ for $1\le j\le k$ and $r_k$ are defined by multiple integrals depending on the boundary $\varphi$. The approximation simply consists in truncating the expansion. 
Let us note that the first term corresponds in fact to the tangent approximation of Strassen \cite{Strassen} and Daniels \cite{Daniels}, see also \cite{Ferebee-1982}. The convergence of the series and the error bounds
can be made precise if the curved boundary is wholly concave or wholly convex. 
Many studies concern such a series expansion, let us mention a few: Ferebee \cite{Ferebee-1983}, Ricciardi et al. \cite{Ricciardi-al-1984}; Giorno et al. \cite{Giorno-al-1989},
Sacerdote and Tomassetti \cite{Sacerdote-Tomassetti-1996} to deal with more general diffusion processes. The numerical approach proposed in \cite{Buonocore-1987} seems to be particularly efficient. In \cite{Nardo-al-2001}, the authors proposed a comparison between the approximation developed by Durbin and an other numerical resolution of the Volterra equation for Gaussian processes.

One method in approximating the passage time of a Brownian motion, or even of a quite general diffusion, through a curved boundary is to replace the initial boundary by an other one which is close and which leads to an explicit expression of the hitting time probability. Such method permits to obtain some bounds. It was first introduced for the Brownian motion in \cite{Borovkov-Novikov-2005}, 
and applied for instance to piecewise continuous boundaries \cite{Wang-Potzelberger-2007}.
Finally an other method consists in writing the p.d.f of the hitting time as the expectation of a particular functional of a three-dimensional Brownian bridge. It suffices then to approximate this expectation through a Monte Carlo method \cite{Ichiba-2011}.
 
\subsubsection*{Approximation of the first passage time.}

All methods described so far concern the approximation of the pdf. It can be of particular interest to simulate directly the first passage time $\tau_\varphi$ or to compute the probability for the hitting time to be smaller than some given $T>0$, without computing the pdf. The solution consists in using a time discretization of the Brownian motion on $[0,T]$. The time interval is then split into $n$ small intervals of the kind $[(k-1)T/n,kT/n]$, with $1\le k\le n$. 
It is therefore possible just to simulate the hitting time of the corresponding Euler scheme. Of course this should upper-bound the stopping time. One solution to overcome this problem is to improve the algorithm by shifting the boundary to reach: we stop the Euler scheme as soon as it exits from a suitable smaller domain. Let us note that this general procedure can also be applied to diffusion processes. It has been first introduced for geometrical Brownian motion (finance) in \cite{Broadie-Glasserman-Kou-1997} and then extended to general diffusions with \emph{nice} coefficients in \cite{Gobet-Menozzi-10}. 

An other method in order to improve the approximation of the hitting time consists in testing, at each endpoint $kT/n$, if the event $B_{kT/n}<\varphi(kT/n)$ is satisfied and if the Brownian path on the small intervals,
conditionally on its value at the end point, hits the curved boundary. This method can also be applied to diffusions and needs therefore precise asymptotics of hitting probabilities for pinned diffusions. A first important study in that direction is \cite{Gobet-2000} where the coefficients of the diffusion are frozen at the starting point on each small interval leading to asymptotics of the probabilities. Nevertheless the method can become onerous if the observed time interval $[0,T]$ is large and sometimes gives incorrect asymptotics: it has been pointed out, by the numerical treatment of some precise examples, that the approximations produced by this method can be far from the true ones. See for this point Giraudo and Sacerdote \cite{Giraudo-Sacerdote-1999} (O.U. process and Feller model),
who also suggest some formulas for the computation of the crossing probability, see also \cite{Giraudo-al-2001}. Baldi and Caramellino \cite{Baldi-Caramellino-2002} presents precise asymptotics for general pinned diffusions which permits to improve the approximation of hitting times. Such results can be developed further in the particular gaussian framework for any dimension \cite{Caramellino-Pacchiarotti-Salvadei-2015}.

\subsubsection*{A new algorithm.}

The aim of this study is to present a new method of approximation of $\tau_\varphi$. Let us explain intuitively the simulation procedure. If $\varphi$ is an increasing curve with $\varphi(0)>0$, then the Brownian motion needs to successively cross a sequence of imaginary horizontal lines before hitting the boundary. The first line to cross corresponds to the value $\varphi(0)$ and needs a random time denoted by $\mathcal{T}_1$. At that time, the value of the curved boundary is $\varphi(\mathcal{T}_1)$. The Brownian motion therefore needs to cross this second horizontal line, it shall happen at time $\mathcal{T}_2$ and the new horizontal line to cross becomes $\varphi(\mathcal{T}_2)$ and so on... Figure \ref{fig:description} (left) illustrates this procedure. The sequence of stopping times $(\mathcal{T}_n)$ converges towards $\tau_\varphi$ and will be used in order to obtain an approximation. We shall introduce a stopping procedure in this sequence of random times which depends on a small parameter $\epsilon$ associated to the error size of the approximation: the sequence is stopped as soon as the distance between two successive horizontal lines is smaller than $\epsilon$. The outcome of the algorithm corresponds therefore to a random variable $\tau_\varphi^\epsilon$ which can be exactly simulated
and such that $\tau_\varphi^{\epsilon}$ converges toward $\tau_\varphi$ in distribution as $\epsilon$ tends to $0$.

The algorithm can be modified when the curved boundary $\varphi$ does not satisfy the monotonic property anymore. In such a slightly different context, it suffices to tilt the successive imaginary horizontal lines in such a way that the common slope corresponds to $\inf_{t\ge 0}\varphi'(t) $, see Figure \ref{fig:description} (right).

\begin{figure}[h]
\centerline{\includegraphics[width=0.48\textwidth]{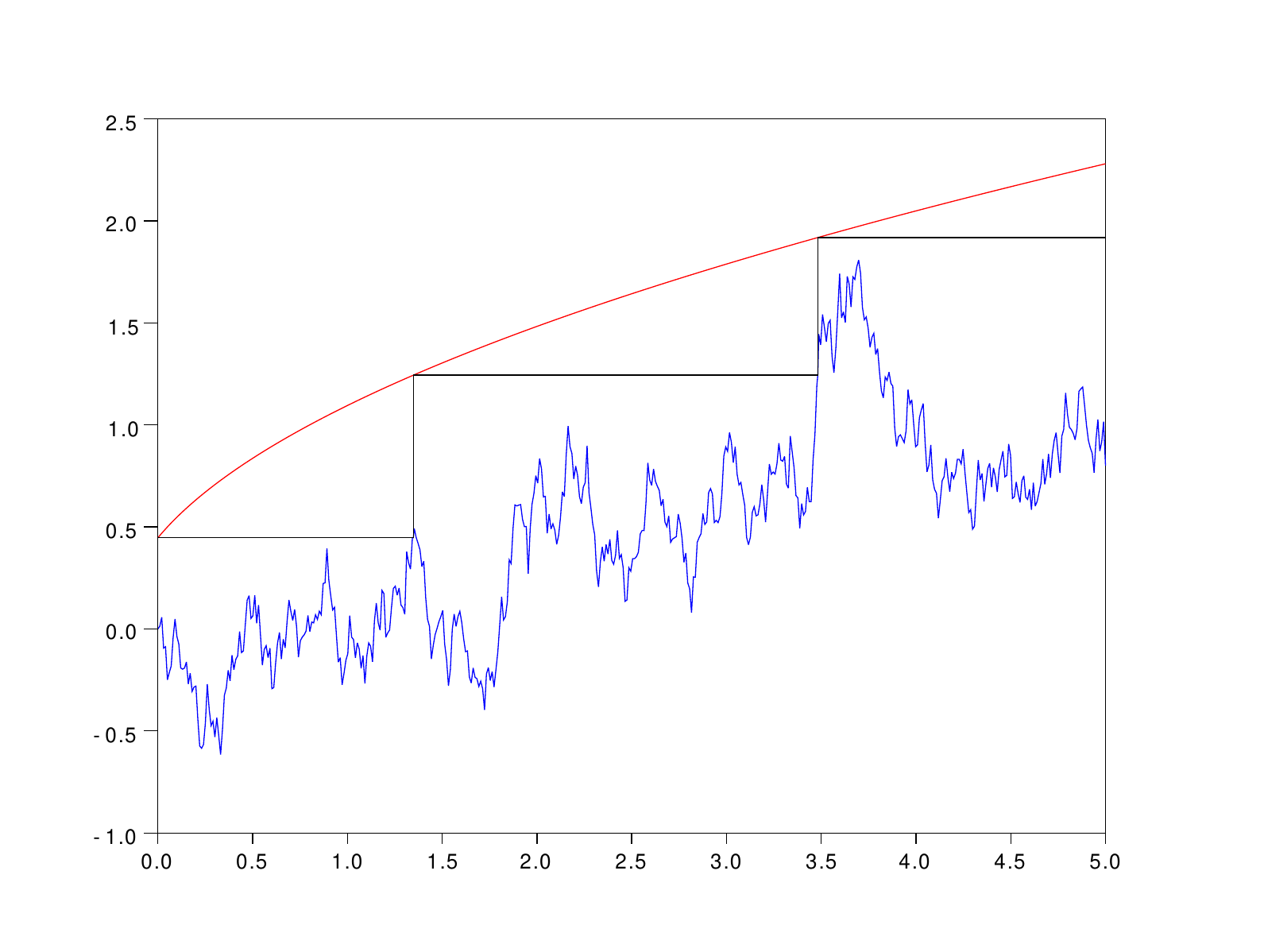}
\includegraphics[width=0.5\textwidth]{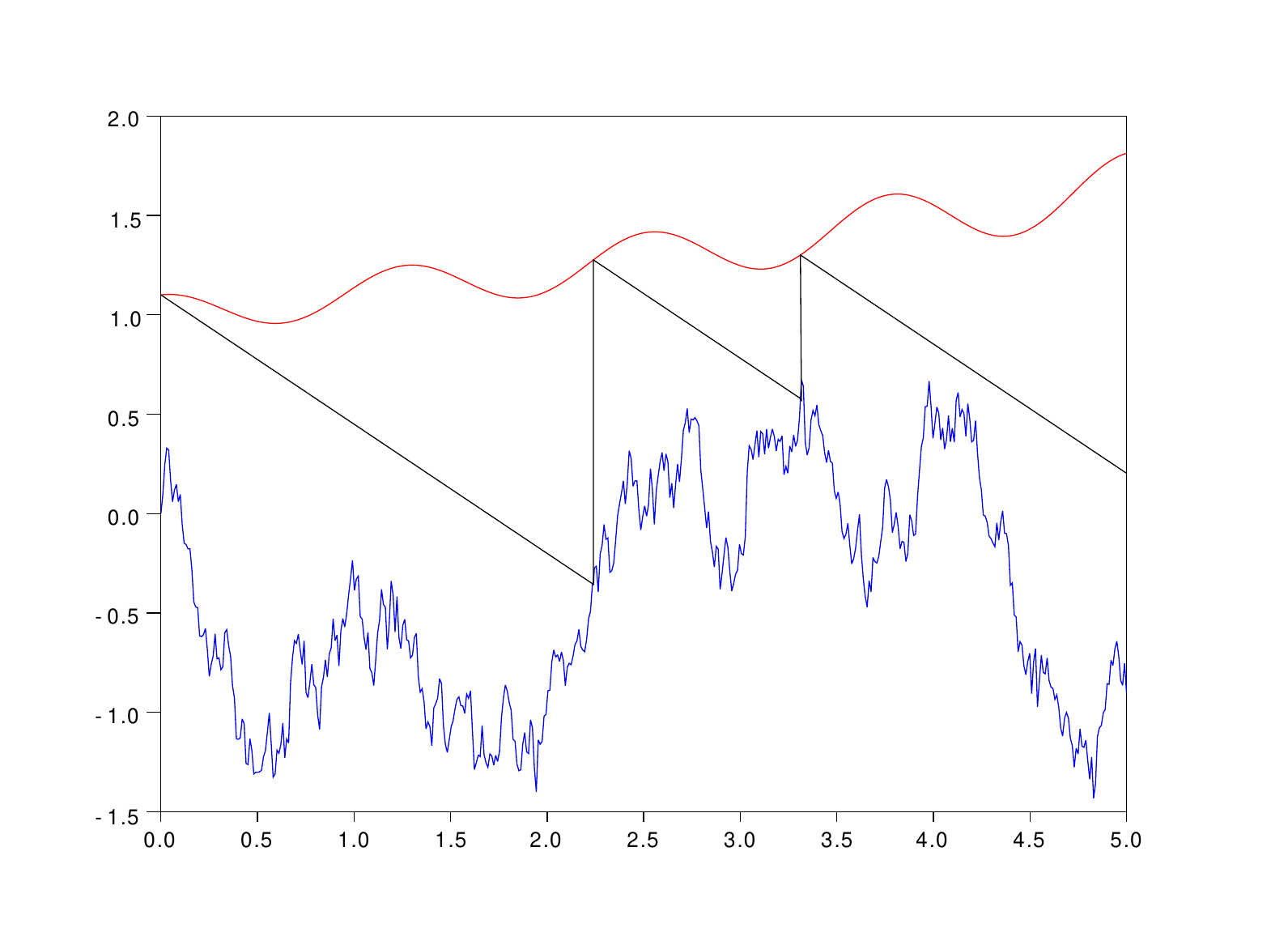}}
\caption{Illustration of the algorithm for an increasing boundary $\varphi$ with its associated successive horizontal lines (left) and for a general boundary (right).}
\label{fig:description}
\end{figure}

To sum up, two different families of sequences will be developed  and the associated 
convergence rates are estimated. The first algorithm developed in Section \ref{sec:incre} concerns increasing curved boundaries and the second one, Section \ref{sec:gen}, permits 
us to deal with quite general boundaries provided that its derivative is bounded. In the last section, we present different examples in order to illustrate the algorithm efficiency. 

\section{First-passage time to non-decreasing boundaries}\label{sec:incre}
Let us assume that the boundary $\varphi$ satisfies \eqref{hypo} and that the following additional conditions hold
\begin{equation}
\label{hypo+}
\varphi:\mathbb{R}_+\to\mathbb{R}\quad\mbox{is a non-decreasing }\ \mathcal{C}^1\mbox{-continuous function,}\tag{H2}
\end{equation}
\begin{equation}
\label{hypo++}
2\varphi'(t)\sqrt{1+t}\le 1,\quad \forall t\ge 0.\tag{H3}
\end{equation}
We introduce the algorithm associated to the hitting time $\tau_\varphi$ defined by \eqref{def:tau}.
\begin{Algo}
\label{algo1}
Let $\epsilon>0$ be a small parameter and $(G_n)_{n\ge 0}$ a sequence of independent standard Gaussian distributed random variables.\\ Initialization: $T_0=0$, $T_1=(\varphi(0)/G_0)^2$ and $\mathcal{N}_\epsilon=1$.\\
While $\varphi(T_1)-\varphi(T_0)>\epsilon$ do: 
\begin{eqnarray}\label{eq:algo1}
\left\{\begin{array}{l}
(T_0,T_1)\leftarrow \Big(T_1,T_1+(\varphi(T_1)-\varphi(T_0))^2/G_{\mathcal{N}_\epsilon}^2\Big)\\
\mathcal{N}_\epsilon\leftarrow \mathcal{N}_\epsilon+1.
\end{array}\right.
\end{eqnarray}
Outcome: $\tau_{\varphi}^\epsilon\leftarrow T_1$ and $\mathcal{N}_\epsilon$.
\end{Algo}
Let us just note that Algorithm \ref{algo1} is very simple to use since each step only requires one Gaussian distributed random variable. Moreover it is a approximation of the first-passage time:
\begin{thm}\label{thm1}\begin{enumerate}
\item Let us assume that the boundary function $\varphi$ satisfies \eqref{hypo}, \eqref{hypo+} and \eqref{hypo++} then the random variable $\tau_\varphi^\epsilon$ defined in Algorithm \ref{algo1} converges in distribution towards $\tau_\varphi$ defined by \eqref{def:tau} as $\epsilon$ tends to zero. More precisely
\begin{equation}\label{eq:thm:1}
F_\epsilon(t-\epsilon)-\frac{3\sqrt{\epsilon}}{\sqrt{2\pi}}\le F(t)\le F_\epsilon(t),\quad\mbox{for any}\quad t\ge \epsilon,
\end{equation}
where $F$ (resp. $F_\epsilon$) is the cumulative distribution function of $\tau_\varphi$ (resp. $\tau_\varphi^\epsilon$).
\item There exists a constant $C>0$ such that the random number of iterations $\mathcal{N}_\epsilon$ defined in Algorithm \ref{algo1} satisfies:
\begin{equation}
\label{eq:thm:conv1}
\mathbb{E}[\mathcal{N}_\eps]\le C\sqrt{|\log \epsilon|}.
\end{equation}
 \end{enumerate}
\end{thm}
The parameter $\epsilon$ describes the precision of the approximation. The number of steps in the Algorithm \ref{algo1} is very small (even smaller than usual results obtained for algorithms based on random walks on spheres, which are close to Algorithm \ref{algo1}, see \cite{muller})  : in fact the constant appearing in \eqref{eq:thm:conv1} can be explicitly computed: for any constant $0<\kappa<1/2$, there exists $\epsilon_0(\kappa)>0$ such that \eqref{eq:thm:conv1} is satisfied as soon as
 $\epsilon<\epsilon_0$, with the particular constant 
 \[
 C=\frac{1}{m\kappa},\quad m=\log(4)+\frac{2\sqrt{2}}{\sqrt{\pi}}\, \mu
\]
and
\begin{equation}
\label{eq:def:mu}
\mu=\int_0^\infty(\log|x|)\,e^{-x^2}\,dx.
\end{equation}
The proof of Theorem \ref{thm1} is based on a main argument developed in the following proposition: each step of Algorithm \ref{algo1} has to be related to a particular part of the Brownian paths before hitting the boundary. 
\begin{prop}
\label{prop:temps} Let $(B_t,\, t\ge 0)$ be a standard one-dimensional Brownian motion. We define the following sequence of stopping times: $s_0=\mathcal{T}_0=0$ and for any $n\ge 1$:
\begin{equation}\label{eq:prop:temps}
s_n:=\inf\Big\{t\ge 0:\ B_{t+\mathcal{T}_{n-1}}=\varphi(\mathcal{T}_{n-1})\Big\}\quad \mbox{and} \quad \mathcal{T}_n:=s_1+\ldots+s_n,
\end{equation}
where the function $\varphi$ satisfies \eqref{hypo}, \eqref{hypo+} and \eqref{hypo++}.
Then the following properties hold:
\begin{enumerate}
\item $(\mathcal{T}_n)_{n\ge 0}$ is a non-decreasing sequence which almost surely converges towards $\tau_\varphi$.
\item Let $n\ge 1$, then the probability distribution of $s_{n+1}$ given the $\sigma$-algebra $\mathcal{F}_n:=\sigma\{  s_1\ldots,s_n\}$ is identical as $( \varphi(\mathcal{T}_{n})-\varphi(\mathcal{T}_{n-1}) )^2/G_n^2$ where $(G_n)_{n\ge 0}$ is a sequence of independent standard Gaussian random variables. Moreover $s_1\stackrel{(d)}{=}(\varphi(0)/G_0)^2$.
\item Let $\mathcal{M}_\epsilon:=\inf\{n\ge 1:\ \varphi(\mathcal{T}_n)-B_{\mathcal{T}_n}\le \epsilon\}$, then $\mathcal{T}_{\mathcal{M}_\epsilon}$ and $\tau_\varphi^\epsilon$, defined in Algorithm \ref{algo1}, are identically distributed, so are $\mathcal{M}_\epsilon$ and $\mathcal{N}_\epsilon$.
\end{enumerate}
\end{prop}
Let us note that the mean of each random variable $s_n$ defined by \eqref{eq:prop:temps} is infinite since $\mathbb{E}[G^{-2}]=+\infty$ where $G$ is a standard Gaussian variable. Proposition \ref{prop:temps} suggests that the first-passage time can be obtained as a sum of positive random variables of infinite average, we easily deduce $\mathbb{E}[\tau_\varphi]=+\infty$. In the particular case of increasing boundaries $\varphi$, the sum has infinitely many terms.
\begin{proof}[Proof of Proposition \ref{prop:temps}] ~\\\textbf{Step 1.}
By construction, the sequence $(\mathcal{T}_n)_{n\ge 0}$ is non-decreasing and non-negative: it converges almost surely to $\mathcal{T}_\infty$. Since $\varphi$ is a non-decreasing boundary, $\mathcal{T}_n\le \tau_\varphi$ for any $n\ge 0$. In particular $\mathcal{T}_\infty$ is less than $\tau_\varphi$ which is a finite stopping time due to the law of the iterated logarithm, see \eqref{hypo} followed by discussion. Consequently, the random variable $B_{\mathcal{T}_\infty}$
is well defined. 
Since $\varphi$
 is non-decreasing, we get $B_{\mathcal{T}_n}=\varphi(\mathcal{T}_{n-1})
 $ for any $n\ge 1$.  
 Taking the large $n$ limit leads to $B_{\mathcal{T}_\infty} = \varphi(\mathcal{T}_\infty)$, the Brownian paths and the function \(\varphi\) being continuous. We deduce that $\mathcal{T}_\infty=\tau_\varphi$.  
\\
\textbf{Step 2.} Let us first consider the stopping time $s_1$. Using the reflection principle of the Brownian paths and a scaling property, we obtain:
\begin{align*}
\mathbb{P}(s_1> t)&=\mathbb{P}\Big( \sup_{0\le u\le t}B_u < \varphi(0) \Big)=\mathbb{P}(|B_t|<\varphi(0))\\
&=\mathbb{P}(B_1^2< \varphi(0)^2/t)=\mathbb{P}(\varphi(0)^2/G_0^2> t),\quad t\ge 0.
\end{align*}
The general $n$-th case can be proven using similar arguments combined with the Markov property of the Brownian motion:
\begin{align*}
\mathbb{P}(s_{n+1}> t\vert \mathcal{F}_n)&=\mathbb{P}\Big( \sup_{\mathcal{T}_n\le u\le \mathcal{T}_n+ t}B_u < \varphi(\mathcal{T}_n) \Big| \mathcal{F}_n\Big)\\
&=\mathbb{P}\Big( \sup_{0\le u\le  t}B_{u+\mathcal{T}_n}-B_{\mathcal{T}_n} < \varphi(\mathcal{T}_n)-\varphi(\mathcal{T}_{n-1}) \Big| \mathcal{F}_n\Big)\\
&=\mathbb{P}\Big( \sup_{0\le u\le  t}\tilde{B}_u < \varphi(\mathcal{T}_n)-\varphi(\mathcal{T}_{n-1}) \Big| \mathcal{F}_n\Big),
\end{align*}
where $\tilde{B}$ is a Brownian motion independent of $\mathcal{F}_n$.\\
\textbf{Step 3.} Using the results developed in Step 2, we observe that $(s_n)_{n\wedge \mathcal{M}_\epsilon}$ and the sequence of values $T_1$, defined in Algorithm \ref{algo1}, have the same distribution. It is therefore obvious that $\mathcal{T}_{\mathcal{M}_\epsilon}$ and $\tau_\varphi^\epsilon$ are identically distributed. Indeed the stopping time can be rewritten as follows:
\begin{equation}\label{eq:reecr}
\mathcal{M}_\epsilon=\inf\{ n\ge 1: \varphi(\mathcal{T}_n)-\varphi(\mathcal{T}_{n-1})\le \epsilon \}.
\end{equation}
\end{proof}
\begin{proof}[Proof of Theorem \ref{thm1}] ~\\
\textbf{Step 1.} Let us recall that $\mathcal{T}_n$ is defined by \eqref{eq:prop:temps}. By Proposition \ref{prop:temps}, $\mathcal{T}_n\le \tau_\varphi$ for any $n\ge 0$ and in particular $\mathcal{T}_{\mathcal{M}_\epsilon}\le \tau_\varphi$. Hence
\[
\mathbb{P}(\mathcal{T}_{\mathcal{M}_\epsilon}\le t)\ge \mathbb{P}(\tau_\varphi\le t),\quad \forall t\ge 0.
\]
Since $\tau_\varphi^\epsilon$ has the same distribution as $\mathcal{T}_{\mathcal{M}_\epsilon}$, we obtain
\begin{equation}
\label{eq:repar-facil}
F_\epsilon(t)\ge F(t),\quad\forall t\ge 0,
\end{equation}
where $F_\epsilon$ and $F$ are the associated cumulative distribution functions. Let us now prove the second bound in \eqref{eq:thm:1}. For $t\ge \epsilon$, 
\begin{align}\label{eq:etap1}
F_\epsilon(t-\epsilon)&=\mathbb{P}(\tau_\varphi^\epsilon\le t-\epsilon)=\mathbb{P}(\mathcal{T}_{\mathcal{M}_\epsilon}\le t-\epsilon)\nonumber\\
&\le\mathbb{P}(\mathcal{T}_{\mathcal{M}_\epsilon}\le t-\epsilon,\tau_\varphi>t)+\mathbb{P}(\tau_\varphi\le t)\nonumber\\
&\le \mathbb{P}(|\mathcal{T}_{\mathcal{M}_\epsilon}-\tau_\varphi|>\epsilon)+F(t).
\end{align}
Combining the Markov property of the Brownian motion and the reflection principle leads to
\begin{align*}
P_\epsilon &:=\mathbb{P}(|\mathcal{T}_{\mathcal{M}_\epsilon}-\tau_\varphi|>\epsilon)\le 1-\mathbb{P}\Big(\sup_{0\le u \le \epsilon} B_{\mathcal{T}_{\mathcal{M}_\epsilon}+u}\ge 
\sup_{0\le u \le \epsilon} \varphi(\mathcal{T}_{\mathcal{M}_\epsilon}+u)\Big)\\ 
&\le 1-\mathbb{P}\Big( \sup_{0\le u \le \epsilon} B_{\mathcal{T}_{\mathcal{M}_\epsilon}+u} 
- B_{\mathcal{T}_{\mathcal{M}_\epsilon}}\ge \sup_{0\le u \le \epsilon} \varphi(\mathcal{T}_{\mathcal{M}_\epsilon}+u)- \varphi(\mathcal{T}_{\mathcal{M}_\epsilon}) + \epsilon\Big) \\
&\le 1-\mathbb{P}\Big( \sup_{0\le u \le \epsilon} B_{\mathcal{T}_{\mathcal{M}_\epsilon}+u} 
- B_{\mathcal{T}_{\mathcal{M}_\epsilon}}\ge  \varphi(\mathcal{T}_{\mathcal{M}_\epsilon}+\epsilon)- \varphi(\mathcal{T}_{\mathcal{M}_\epsilon}) + \epsilon\Big) \\
&\le 1-\mathbb{P}\Big(  \sup_{0\le u \le \epsilon}\tilde{B}_u\ge \sup_{\mathcal{T}_{\mathcal{M}_\epsilon}\le \theta \le \mathcal{T}_{\mathcal{M}_\epsilon}+\epsilon} \varphi'(\theta)\epsilon+\epsilon\Big)\\
&\le 1-\mathbb{P}\Big( |\tilde{B}_\epsilon|\ge \sup_{\mathcal{T}_{\mathcal{M}_\epsilon}\le \theta \le \mathcal{T}_{\mathcal{M}_\epsilon}+\epsilon} \varphi'(\theta)\epsilon+\epsilon\Big).
\end{align*}
Using Hypothesis \eqref{hypo++} and straightforward computations permits us to obtain
\begin{equation}\label{eq:etap2}
\mathbb{P}(|\mathcal{T}_{\mathcal{M}_\epsilon}-\tau_\varphi|>\epsilon)\le 1-\mathbb{P}(|\tilde{B}_\epsilon|\ge 3\epsilon/2)\le 3\sqrt{\frac{\epsilon}{2\pi}}.
\end{equation}
The lower bound in \eqref{eq:thm:1} holds due to both \eqref{eq:etap1} and \eqref{eq:etap2}. \\
\textbf{Step 2.} Let us now focus our attention to the efficiency of this algorithm. 
We need to estimate the number of steps which depends on the small parameter $\epsilon$. Using the third result presented in Proposition \ref{prop:temps} on one hand and \eqref{eq:reecr} on the other hand, we obtain
\[
\mathbb{P}(\mathcal{N}_\epsilon> n)=\mathbb{P}(\mathcal{M}_\epsilon> n)=\mathbb{P}(\varphi(\mathcal{T}_1)-\varphi(\mathcal{T}_0)>\epsilon,\ldots,\varphi(\mathcal{T}_n)-\varphi(\mathcal{T}_{n-1})>\epsilon ).
\]
Hypothesis \eqref{hypo++} implies
\begin{equation}\label{eq:lien}
\mathbb{P}(\mathcal{N}_\epsilon> n)\le \mathbb{P}(s_1>2\epsilon,\ldots, s_n>2\epsilon).
\end{equation}
\textbf{Step 2.1.} Let us first estimate the previous upper-bound. We introduce a sequence of independent  standard Gaussian random variables $(G_n)_{n\ge 0}$ and define
\begin{equation}
\label{eq:def:Z}
X_n=\log(4G_n^2),\quad \Xi_n=\sum_{k=0}^n X_k\quad\mbox{and}\quad Z_n=\sum_{k=0}^n \Xi_k.
\end{equation}
Let us define $\Pi(n,\epsilon):=\mathbb{P}(s_n>2\epsilon)$. 
By Proposition \ref{prop:temps}, we know that the random variables $s_{n+1}$ 
are related to $G_n$ and therefore
\begin{align*}
 \Pi(1,\epsilon)&=\mathbb{P}(2\epsilon G_0^2<\varphi(0)^2)=\mathbb{P}\Big(\log(4G_0^2)<-\log(\epsilon)+\log(2)+2\log\varphi(0)\Big)\\
 &=\mathbb{P}\Big( Z_0< -\log(\epsilon)+\log(2)+2\log\varphi(0)\Big).
 \end{align*}
 Let us prove that, for $n\ge 1$, we have the general formula:
 \begin{equation}
 \label{eq:genform}
 \Pi(n,\epsilon)\le \mathbb{P}\Big(Z_{n-1}<-\log(\epsilon)+(2n-1)\log(2)+(2n)\log \varphi(0)\Big).
 \end{equation}
By Proposition \ref{prop:temps}, we have for $n\ge 2$,
 \begin{align}\label{eq:1}
 \Pi(n,\epsilon)=\mathbb{P}\Big( ( \varphi(\mathcal{T}_{n-1})-\varphi(\mathcal{T}_{n-2}) )^2>2\epsilon G_{n-1}^2 \Big).
\end{align}
Since $\varphi$ is a non decreasing function satisfying Hypothesis \eqref{hypo++}, the following upper-bound holds for $n\ge 2$:
\begin{equation}\label{eq:2}
\varphi(\mathcal{T}_{n-1})-\varphi(\mathcal{T}_{n-2}) \le \frac{\mathcal{T}_{n-1}-\mathcal{T}_{n-2} }{2\sqrt{1+\mathcal{T}_{n-2}}}\le\frac{s_{n-1}}{2\sqrt{1+s_{n-2}}}.
\end{equation}
Hence for $n=2$, \eqref{eq:1} and \eqref{eq:2} imply
\begin{align*}
\Pi(2,\epsilon)&\le \mathbb{P}\Big(\frac{s_{1}^2}{2^2}>2\epsilon G_1^2\Big)=\mathbb{P}\Big(\epsilon (2 G_1^2)(2G_0^2)^2<\varphi(0)^4\Big)\\
&=\mathbb{P}(2X_0+X_1<-\log(\epsilon)+3\log(2)+4\log\varphi(0))\\
&=\mathbb{P}(Z_1<-\log(\epsilon)+3\log(2)+4\log\varphi(0)).
\end{align*}
Using the lower-bound $1+s_{n-1}\ge s_{n-1}$ and similar arguments as those developed previously, the general case is expressed as follows:
\begin{align*}
\Pi(n,\epsilon)&\le \mathbb{P}\Big( \frac{s_{n-1}^2}{2^2(1+s_{n-2})}>2\epsilon G_{n-1}^2 \Big)\le \mathbb{P}\Big( \frac{s_{n-2}^3}{2^22^4s_{n-3}^2}>2\epsilon G_{n-1}^2G_{n-2}^4 \Big)\\
&\le \mathbb{P}\Big(\frac{s_{2}^{n-1}}{s_1^{n-2}}>\epsilon 2 2^2 2^4\ldots 2^{2(n-2)}G_{n-1}^2G_{n-2}^4\ldots G_{2}^{2(n-2)}\Big)\\
&\le\mathbb{P}\Big(\Big(\frac{s_{1}^2}{2^2G_1^2}\Big)^{n-1}\frac{1}{s_1^{n-2}}>\epsilon 2 2^2 2^4\ldots 2^{2(n-2)}G_{n-1}^2G_{n-2}^4\ldots G_{2}^{2(n-2)}\Big)\\
&\le\mathbb{P}\Big(\varphi(0)^{2n}>\epsilon 2 2^2 2^4\ldots 2^{2(n-1)}G_{n-1}^2G_{n-2}^4\ldots G_{0}^{2n}\Big)\\
&\le \mathbb{P}\Big(Z_{n-1}<-\log(\epsilon)+(2n-1)\log(2)+(2n)\log\varphi(0)\Big). 
\end{align*}
\textbf{Step 2.2.} By \eqref{eq:lien} and the arguments developed in Step 2.1, we obtain
\[
\mathbb{P}(\mathcal{N}_\epsilon>n)\le \mathbb{P}(s_n>2\epsilon)\le \mathbb{P}(Z_{n-1} - \mathbb{E}Z_{n-1}<\eta(\epsilon,n)- \mathbb{E}Z_{n-1}),
\]
where \[
\eta(\epsilon,n):=-\log(\epsilon)+(2n-1)\log(2)+(2n)\log\varphi(0).
\]
Let us observe that, for any $n\ge 0$,
$m:=\mathbb{E}[X_n]=\log(4)+\frac{2\sqrt{2}}{\sqrt{\pi}}\mu>0$ where $\mu$ is defined by \eqref{eq:def:mu}. Hence 
\[
\mathbb{E}[Z_n]=\sum_{k=0}^n\mathbb{E}[\Xi_n]=\sum_{k=0}^n\sum_{j=0}^k\mathbb{E}[X_j]=m \sum_{k=0}^n(k+1)=\frac{m(n+1)(n+2)}{2}.
\]
Thus, for \(n\) large enough, \(\eta(\epsilon,n)- \mathbb{E}Z_{n-1} < 0\).
Introducing $d_n:=|mn(n+1)/2-\eta(\epsilon,n)|$, we observe that, for any $0<\kappa<1/2$ there exists $\aleph(\kappa,\epsilon)\in\mathbb{N}$ such that $d_n>mn^2(1/2-\kappa)$ for $n$  sufficiently large that is $n\ge \aleph(\kappa,\epsilon)$. After straightforward computations, we can choose 
\begin{equation}\label{eq:eta}
\aleph(\kappa, \epsilon):=\Big\lfloor\sqrt{\frac{|\log(2\epsilon)|}{m\kappa}}+\Big| \frac{1}{2\kappa}-\frac{\log(2\varphi(0))}{m\kappa} \Big|\Big\rfloor+1.
\end{equation}
Markov's inequality leads to
\begin{equation}\label{eq:mark}
\mathbb{P}(\mathcal{N}_\epsilon>n)\le \mathbb{P}(|Z_{n-1}-\mathbb{E}[Z_{n-1}]|>d_n)\le \frac{\mathbb{E}[(Z_{n-1}-\mathbb{E}[Z_{n-1}])^4]}{d_n^4}.
\end{equation}
Let us note that $\overline{X}_j:=X_j-m$ are i.i.d. random variables with finite moments of any order. We denote $m_k:=\mathbb{E}[\overline{X}_j^k]$. Therefore we obtain
\begin{align}
\label{eq:mom}
\mathcal{Z}_{n-1}&:=\mathbb{E}[(Z_{n-1}-\mathbb{E}[Z_{n-1}])^4]=\mathbb{E}\Big[ \Big(\sum_{k=0}^{n-1}\sum_{j=0}^k \overline{X}_j\Big)^4 \Big]=\mathbb{E}\Big[ \Big(\sum_{j=0}^{n-1}(n-j) \overline{X}_j\Big)^4 \Big]\nonumber\\
&=\sum_{j=0}^{n-1}(n-j)^4m_4+2\sum_{0\le j<k\le n-1}(n-j)^2(n-k)^2m_2^2\nonumber\\
&\le \frac{m_4}{30}\, n(n+1)(6n^3+9n^2+n-1)+\frac{m_2^2}{36}\, n^2(n+1)^2(2n+1)^2.
\end{align}
Hence, there exist a constant $C_0>0$ such that $\mathbb{E}[(Z_{n-1}-\mathbb{E}[Z_{n-1}])^4]\le C_0 n^6$. Combining the previous inequality with \eqref{eq:eta} and \eqref{eq:mark} leads to
\[
\mathbb{P}(\mathcal{N}_\epsilon>n)\le \frac{C_0}{m^4(1/2-\kappa)^4}\,\frac{1}{n^2},\quad \mbox{for}\quad n\ge \aleph(\kappa,\epsilon).
\]
Consequently, the following upper-bound holds
\[
\mathbb{E}[\mathcal{N}_\epsilon]=\sum_{n\ge 0}\mathbb{P}(\mathcal{N}_\epsilon>n)\le \aleph(\kappa,\epsilon)+\frac{C_0}{m^4(1/2-\kappa)^4}\sum_{n\ge \aleph(\kappa,\epsilon)}\frac{1}{n^2}.
\]
In order to conclude, it suffices to note that $\aleph(\kappa,\epsilon)\to \infty$ as $\epsilon\to 0$, the second term in the previous inequality therefore becomes small as $\epsilon\to 0$: the leading term is finally $\aleph(\kappa,\epsilon)$ which is equivalent to $\sqrt{|\log(2\epsilon)|/(m\kappa)}$ by \eqref{eq:eta}.
\end{proof}

\section{First-passage time to boundaries with bounded derivative}
\label{sec:gen}
The algorithm presented in Section \ref{sec:incre} is simple to achieve (it only requires independent Gaussian random variables) 
and efficient: the averaged number of steps is of the order $\sqrt{|\log\epsilon|}$ where $\epsilon$ stands for the small parameter 
appearing in the rejection sampling (see Theorem \ref{thm1}). In order to apply Algorithm \ref{algo1} the curved boundary, the 
Brownian motion is going to hit, has to satisfies suitable conditions: \eqref{hypo}, \eqref{hypo+} and \eqref{hypo++}. 
Asking for the monotonicity of the function $\varphi$ is quite restrictive, that's why we present an extension of the algorithm 
which is of course less efficient (even if the average number of steps is still very small) but which permits us to deal with more general boundaries. 
Let us introduce the following assumption: there exist two constants $\rho_+>0$ and $\rho_->0$ such that
\begin{align}
\label{hypo-new}
\varphi:\mathbb{R}_+\to \mathbb{R}\ &\mbox{is a}\ \mathcal{C}^1\mbox{-continuous function satisfying}\nonumber\\ &\sup_{t\ge 0}\varphi'(t)\le \rho_+\quad\mbox{and}\quad \inf_{t\ge 0}\varphi'(t)\ge -\rho_-.
\tag{H4}
\end{align}
For such boundaries, we present an algorithm which permits us,
for any \(K\in\mathbb{R}^+\), to approximate the hitting time 
$\tau_\varphi^K = \tau_\varphi\wedge K$, where \(\tau_\varphi\) is defined 
in \eqref{def:tau}. 
Let us introduce some notations: the \emph{inverse Gaussian distribution} of parameters $\mu>0$ and $\lambda>0$ will be denoted by $I(\mu,\lambda)$ and \ is defined by its
the probability distribution function:
\[
f(x)=\sqrt{\frac{\lambda}{2\pi x^3}}\,\exp-\Big\{\frac{\lambda(x-\mu)^2}{2\mu^2x}\Big\}\ \mathbbm{1}_{\{x\ge 0\}}.
\] 
\begin{Algo}
\label{algo2}
Let $\epsilon>0$ be a small parameter and $r>\rho_-$ where $\rho_-$ is defined in \eqref{hypo-new}.\\
Initialization: $(T,H)=(0,\varphi(0))$ and $\mathcal{N}_{\epsilon,K}=0$.\\
While $H>\epsilon$ and \(T < K\), simulate $\hat{G}$ an inverse Gaussian random variable with distribution $I(H/r,H^2)$ and do:
\begin{align}\label{eq:algo2}
\left\{\begin{array}{l}
H\leftarrow\varphi(T+\hat{G})-\varphi(T)+
r\,
\hat{G},\\
T\leftarrow \hat{G}+T,\\
\mathcal{N}_{\epsilon,K}\leftarrow
\mathcal{N}_{\epsilon,K}+1.
\end{array}\right.
\end{align}
Outcome: $\tau^{\epsilon,K}_\varphi\leftarrow T\wedge K$ and $\mathcal{N}_{\epsilon,K}$.
\end{Algo}
Algorithm \ref{algo2} is quite simple, it only requires the simulation of inverse Gaussian distributed random variables. Let us recall the following scaling property: if $\hat{G}\sim I(H/r,H^2)$ then $H\hat{G}/r\sim I(1,r H)$. Moreover $\frac{(r\hat{G}-H)^2}{\hat{G}}$ is Chi-squared distributed with one degree of freedom (the square of a standard Gaussian random variable). In order to simulate an inverse Gaussian random variable, we suggest to use the algorithm introduced by Michael, Schucany and Haas (see \cite{MSH} or \cite[p.~149]{devroye}). Let us now state the efficiency of Algorithm \ref{algo2}. The inverse Gaussian distribution does not permit us to argue in a similar way as in Section \ref{sec:incre}. That's why we are going to use the general potential theory in order to upper-bound the averaged number of steps. This kind of arguments was already introduced in convergence results associated to the Random Walk on Spheres algorithm which permits 
the approximation of the solution of the Dirichlet problem, see for instance \cite{muller}. 
\begin{thm}\label{thm2}\begin{enumerate}
\item Let us assume that the boundary function $\varphi$ satisfies \eqref{hypo-new} then the random variable $\tau_\varphi^{\epsilon,K}$ defined in Algorithm \ref{algo2} converges in distribution towards $\tau_\varphi^K=\tau_\varphi\wedge K$ where $\tau_\varphi$ is defined by \eqref{def:tau} as $\epsilon$ tends to zero. More precisely
\begin{equation}\label{eq:thm:21}
F_{\epsilon,K}(t-\epsilon)-(1+\rho)
\sqrt{\frac{2\epsilon}{\pi}}\le F_K(t)\le F_{\epsilon,K}(t),\quad\mbox{for any}\quad t\ge \epsilon,
\end{equation}
where $F_K$ (resp. $F_{\epsilon,K}$) is the cumulative distribution function of $\tau_\varphi^K$ (resp. $\tau_\varphi^{\epsilon,K}$).
\item There exist positive constants $a$, $b$, $\kappa_0$, $\kappa_1$ and $\epsilon_0$ such that: for any $\rho_+\le \kappa_0$ and any $(K,r)$ satisfying $(r+\kappa_0)K\le \kappa_1$, the random number of iterations $\mathcal{N}_{\epsilon,K}$ defined in Algorithm \ref{algo2} satisfies the following upper bound
\begin{equation}
\label{eq:thm:conv12}
\mathbb{E}[\mathcal{N}_{\eps,K}]\le (a+br)|\log \epsilon|,\quad \forall  \epsilon\le \epsilon_0.
\end{equation}
\item For non increasing functions $\varphi$: there exists two positive constants $a$ and $\epsilon_0$ such that
\begin{equation}
\label{eq:thm:conv13}
\mathbb{E}[\mathcal{N}_{\eps,K}]\le ar^2K|\log \epsilon|,\quad \forall  \epsilon\le \epsilon_0.
\end{equation}
 \end{enumerate}
\end{thm}
This theorem is based on the following intermediate statement which is a modification of Proposition \ref{prop:temps}.
\begin{prop}
\label{prop:2}
Let $(B_t,\ t\ge 0)$ be a standard one-dimensional Brownian motion. We introduce the following stopping times: $s_0=\mathcal{T}_0^K=0$ and for any $n\ge 1$:
\begin{equation}
\label{eq:prop:temps2}
s_n:=\inf\Big\{ t\ge 0:\ B_{t+\mathcal{T}_{n-1}^K}=
\varphi(\mathcal{T}_{n-1}^K)-r t \Big\}\quad\mbox{and}\quad \mathcal{T}_n^K:=(s_1+\ldots+s_n)\wedge K,
\end{equation}
where the boundary $\varphi$ satisfies  \eqref{hypo-new}. Then the following properties hold:
\begin{enumerate}
\item $(\mathcal{T}_n^K)_{n\ge 0}$ is a non-decreasing sequence which almost surely converges towards $\tau_\varphi^K$.
\item On the event \(\{s_1 + \cdots + s_n < K\}\), the probability distribution of $s_{n+1}$ given the $\sigma$-algebra $\mathcal{F}_n:=\sigma\{ \mathcal{T}_1^K,\ldots,\mathcal{T}_n^K\}$ is the inverse Gaussian distribution $I(\mathcal{H}_n/r,\mathcal{H}_n^2)$ with
\begin{equation}
\label{eq:def:hn}
\mathcal{H}_n:=\varphi(\mathcal{T}_{n}^K)
-\varphi(\mathcal{T}_{n-1}^K)+r s_n.
\end{equation}
\item Let $\mathcal{M}_\epsilon:=\inf\{n\ge 1:\ \varphi(\mathcal{T}_n^K)-B_{\mathcal{T}_n^K}\le \epsilon\}$,
\(\mathcal{M}^K:=\inf\{n\ge 1,\  \mathcal{T}_n^K = K\}\), 
\(\mathcal{M}_{\epsilon}^K=\mathcal{M}_\epsilon \wedge \mathcal{M}^K\). 
then $\mathcal{T}_{\mathcal{M}_\epsilon^K}$ and $\tau_\varphi^{\epsilon,K}$, defined in Algorithm \ref{algo2}, are identically distributed, so are $\mathcal{M}_\epsilon^K$ and $\mathcal{N}_{\epsilon,K}$.
\end{enumerate}
\end{prop}
\begin{proof}[Proof of Proposition \ref{prop:2}] The first and the third part of the proof are left to the reader. They need similar arguments as those presented in Proposition \ref{prop:temps}. 
 Here the monotonicity property is just replaced by \eqref{hypo-new} which permits us easily to prove that $\mathcal{T}_n^K\le \tau_\varphi^K$.\\
Let us now focus our attention to the second part of the statement. Due to the definition of $s_{n+1}$ and since $\{\mathcal{T}_n^K<K\}$, we get $B_{\mathcal{T}_n^K}=\varphi(\mathcal{T}_{n-1}^K)-rs_n$. Hence, we have 
\begin{align*}
s_{n+1}&=\inf\{t\ge 0:\ B_{t+\mathcal{T}_n^K}-B_{\mathcal{T}_n^K}=\varphi(\mathcal{T}_n^K)-B_{\mathcal{T}_n^K}-r t\}\\
&=\inf\{t\ge 0:\ W_t=\mathcal{H}_n -r t \},
\end{align*}
where $W_t=B_{t+\mathcal{T}_n^K}-B_{\mathcal{T}_n^K}$ is a standard Brownian motion independent of 
$\mathcal{F}_n$ 
and the $\mathcal{F}_n$ adapted r.v. $\mathcal{H}_n$ 
is defined by \eqref{eq:def:hn}. The distribution of $s_{n+1}$ corresponds 
to the distribution of the first passage time of the standard Brownian motion 
with drift at the constant level $\mathcal{H}_n$. The probability distribution is well known (see, for instance 
\cite[p.~197]{karatzas}):
\[
\mathbb{P}(s_{n+1}\in dt|\mathcal{F}_n)=\frac{\mathcal{H}_n}{\sqrt{2\pi t^3}}\,\exp -\Big\{ \frac{(\mathcal{H}_n-r t)^2}{2t} \Big\}\, dt,
\] 
we can consequently identify the inverse Gaussian distribution $I(\mathcal{H}_n/r,\mathcal{H}_n^2)$.
\end{proof}
\begin{proof}[Proof of Theorem 
\ref{thm2}] ~\\
\textbf{Step 1.} We can prove the convergence in distribution of $\tau_\varphi^{\epsilon,K}$ towards $\tau_\varphi^K$ using similar arguments as those presented in the proof of Theorem \ref{thm1}. The upper-bound in \eqref{eq:thm:21} is an adaptation of \eqref{eq:repar-facil} which requires that $\mathcal{T}_n^K\le \tau_\varphi^K$ and that $\tau_\varphi^{\epsilon,K}$ and $\mathcal{T}_{\mathcal{M}_\epsilon^K}$ are identically distributed. These conditions are satisfied, see Proposition~\ref{prop:2}. For the lower-bound in \eqref{eq:thm:21}, we obtain
\[
F_{\epsilon,K}(t-\epsilon)\le \mathbb{P}(|\mathcal{T}_{\mathcal{M}_\epsilon^K}
-\tau_\varphi^K|>\epsilon)+F_{\epsilon,K}(t),
\]
see \eqref{eq:etap1} for the details. Let us note that $|\tau_{\mathcal{M}_\epsilon^K}-\tau_\varphi^K|>\epsilon$ leads to the condition $\tau_{\mathcal{M}_\epsilon^K}<K$. Hence $\mathcal{M}_\epsilon^K=\mathcal{M}_\epsilon$. Using the Markov property, the following bound holds:
\[
\mathbb{P}(|\mathcal{T}_{\mathcal{M}_\epsilon^K}
-\tau_\varphi^K|>\epsilon)\le 1-\mathbb{P}\Big( \sup_{0\le u\le \epsilon}\tilde{B}_u\ge \epsilon+\sup_{0\le u\le \epsilon}\varphi(\mathcal{T}_{
\mathcal{M}_\epsilon}+u)-\varphi(\mathcal{T}_{
\mathcal{M}_\epsilon}) \Big).
\]
Here $(\tilde{B}_t,\, t\ge 0)$ stands for a standard Brownian motion independent of $\mathcal{T}_{\mathcal{M}_\epsilon}$.
Combining Hypothesis \eqref{hypo-new} and the reflection principle of the Brownian motion leads to
\begin{align*}
\mathbb{P}(|\mathcal{T}_{\mathcal{M}_\epsilon}
-\tau_\varphi^K|>\epsilon)&\le 1-\mathbb{P}(|\tilde{B}_\epsilon|\ge \epsilon(1+\rho_+))\le (1+\rho_+)\sqrt{\frac{2\epsilon}{\pi}},
\end{align*}
and consequently to the lower bound \eqref{eq:thm:21}.\\
\textbf{Step 2.} Let us now focus our attention to the averaged number of steps in Algorithm \ref{algo2}, denoted by $\mathcal{N}_{\epsilon,K}$. A rough description of the method: we aim to construct a Markov chain and to describe the associated potential. The classical potential theory then permits us to obtain the announced bound. We introduce the Markov chain $R_n:=(\mathcal{T}_n,\mathcal{H}_n)$ for $n\ge 0$. 
We recall that $\mathcal{T}_n=s_1+\ldots+s_n$ is defined by \eqref{eq:prop:temps2} and $\mathcal{H}_n$ by \eqref{eq:def:hn}. The stopping time $\mathcal{M}^K_\epsilon$ defined in Proposition \ref{prop:2} can also be interpreted as the first time the Markov chain $(R_n,\ n\ge 0)$ goes out of the domain $E:=[0,K]\times ]\epsilon,+\infty]$. \\
Let us consider the function $f(x,y)=\log(y)$, defined on $E$, and denote by $P$ the infinitesimal generator associated to the Markov chain $(R_n)_{n\ge 0}$. By Proposition~\ref{prop:2} and for any $(t,h)\in E$, we obtain
\begin{align*}
Pf(t,h)=\mathbb{E}\Big[\log(\varphi(t
+\hat{G})-\varphi(t)+r \,\hat{G})\Big],
\end{align*}
where $\hat{G}$ is an inverse Gaussian distributed random variable with the following density function:
\[
p(x)=\frac{h}{\sqrt{2\pi x^3}}\exp\left\{ -\frac{(h-r x)^2}{2x} \right\},\quad x\ge 0.
\]
By \eqref{hypo-new},  $\varphi(t+\hat{G})-\varphi(t)\le \rho_+\, \hat{G}$, we get
\begin{align}\label{eq:prem}
Pf(t,h)-f(t,h)&\le \log\Big(1+\frac{\rho_+}{r}\Big)+\mathbb{E}\Big[\log\Big(\frac{
r \hat{G}}{h}\Big)\Big].
\end{align}
Let us find now an explicit upper bound of $Pf-f$. 
Using first the change of variables $u=r x/h$ and secondly $u\mapsto 1/u$, we get
\begin{align}\label{eq:calc}
\mathbb{E}\Big[\log\Big(\frac{r \hat{G}}{h}\Big)\Big]&=\int_0^\infty \log\Big( \frac{r x}{h} \Big)\frac{h}{\sqrt{2\pi x^3}}\exp-\frac{(h-r x)^2}{2x}\, dx \nonumber\\
&=\sqrt{\frac{hr}{2\pi}}\int_0^\infty \frac{\log(u)}{u^{3/2}}\, \exp-\frac{hr(1-u)^2}{2u}\, du\nonumber\\
&=\sqrt{\frac{h r}{2\pi}}\int_1^\infty \frac{(1-u)\log(u)}{u^{3/2}}\exp-\frac{h r(1-u)^2}{2u}\, du.
\end{align}
It is then obvious that $\mathbb{E}\Big[\log\Big(\frac{r\hat{G}}{h}\Big)\Big]< 0$. Let us now give a more precise upper-bound. 
We set $\alpha=h r$, then \eqref{eq:calc} emphasizes that $\mathbb{E}\Big[\log\Big(\frac{r\hat{G}}{h}\Big)\Big]$ only depends 
on the parameter $\alpha$, this dependence being continuous. Let us therefore denote this function $\psi(\alpha)$ 
(see Figure~\ref{fig:MCpsi} below representing $\psi$ obtained with the Monte-Carlo method sample size: $10\,000$).\\
\begin{figure}[t]
\begin{center}
\includegraphics[scale=0.4]{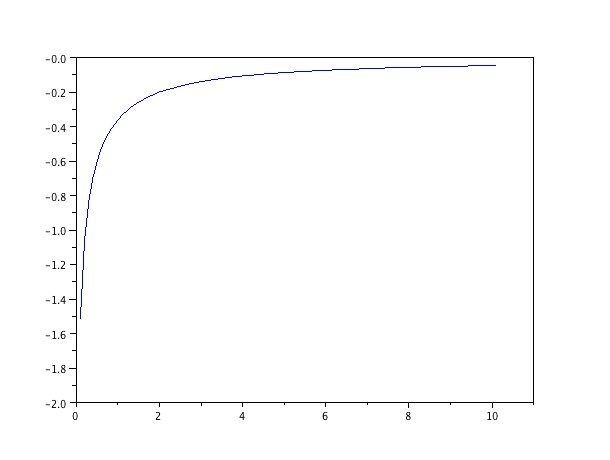}
\end{center}
\caption{Monte Carlo approximation of the function \(\psi\)}\label{fig:MCpsi}
\end{figure}
Simple computations lead to
\begin{align}\label{eq:alpha}
\psi(\alpha)&:=\mathbb{E}\Big[\log\Big(\frac{r \hat{G}}{h}\Big)\Big]=-\sqrt{\frac{\alpha}{2\pi}}\int_0^\infty \frac{u\log(1+u)}{(1+u)^{3/2}}\exp-\frac{\alpha u^2}{2(1+u)}\, du\\
&\le -\sqrt{\frac{\alpha}{2\pi}}\int_0^\infty \frac{u\log(1+u)}{(1+u)^{3/2}}\exp-\frac{\alpha u}{2}\, du
\nonumber\\
&\le
-\frac{1}{\sqrt{2\pi}}\int_0^\infty  \frac{w\log(1+w/\alpha)}{(\alpha+w)^{3/2}}\, \exp-\frac{w}{2}\,dw\nonumber\\
&\le
-\frac{1}{\sqrt{2\pi}}\int_{1/2}^\infty  \frac{w\log(1+w/\alpha)}{(\alpha+w)^{3/2}}\, \exp-\frac{w}{2}\,dw.\nonumber\end{align}Using the inequality $(\alpha+w)\le (1+2\alpha)w$, we get
\begin{align*}
\psi(\alpha)
&  \le -\frac{\log(1+(2\alpha)^{-1})}{(1+2\alpha)^{3/2}\sqrt{2\pi}}\int_{1/2}^\infty \frac{1}{\sqrt{w}}\exp-\frac{w}{2}\,dw
\\
&\le -\frac{\log(1+(2\alpha)^{-1})}{(1+2\alpha)^{3/2}}\,\mathbb{P}(G\ge 1/2),
\end{align*} 
where $G$ is a standard gaussian r.v. and so $\mathbb{P}(G\ge 1/2)\approx 0.3085$\\
We deduce from the previous upper-bound that $\lim_{\alpha\to 0^+}\psi(\alpha)=-\infty$. Moreover the right hand side is a non decreasing function with respect to the variable $\alpha$. Hence 
\begin{equation}\label{eq:num1}
\psi(\alpha)\le -\frac{\log(3/2)}{3\sqrt{3}}\,\mathbb{P}(G\ge 1/2)\approx -0.0241,\quad\mbox{for}\ \alpha\le 1.
\end{equation}
Let us observe what happens for large values of the variable $\alpha$. The Laplace method implies that
\[
\psi(\alpha)\sim -\frac{1}{2\alpha}\quad \mbox{as}\ \alpha\to \infty.
\]
Let us prove now that there exists a constant $c>0$ such that 
\begin{equation}\label{eq:num2}
\psi(\alpha)\le -\frac{c}{\alpha},\quad \mbox{for any}\  \alpha\ge 1. 
\end{equation}
For $\alpha\ge 1$, we get
\begin{align*}
\psi(\alpha)&\le-\sqrt{\frac{\alpha}{2\pi}}\int_0^\infty \frac{u\log(1+u)}{(1+u)^{3/2}}\,\exp-\frac{\alpha u^2}{2}\,du\\
&\le -\sqrt{\frac{\alpha}{2\pi}}\int_0^1 \frac{u\log(1+u)}{(1+u)^{3/2}}\,\exp-\frac{\alpha u^2}{2}\,du.
\end{align*}
Due to the convexity property of the logarithm function ($\log (1+u)\ge \log(2) u$) and the Cauchy-Schwarz inequality, we obtain
\begin{align*}
\psi(\alpha)&\le -\frac{\log(2)}{\alpha 2^{3/2}}\Big(\frac{1}{2}\mathbb{E}[G^2]-\mathbb{E}[G^21_{\{G\ge \sqrt{\alpha}\}}]  \Big)\\
&\le -\frac{\log(2)}{\alpha 2^{3/2}}\Big( \frac{1}{2}-\sqrt{\mathbb{E}[G^4]}\sqrt{\mathbb{P}(G\ge \sqrt{\alpha})} \Big)\\
&\le -\frac{\log(2)}{\alpha 2^{3/2}}\Big( \frac{1}{2}-\frac{\sqrt{3}}{2}\,e^{-\alpha} \Big)\le -\frac{\log(2)}{\alpha 2^{5/2}}(1-\sqrt{3}e^{-1}),\quad\mbox{for}\ \alpha\ge 1.
\end{align*}
We deduce that $\psi(\alpha)\le -c/\alpha$ with $c\approx 0.0445$ when $\alpha\ge 1$. Combining both inequalities \eqref{eq:num1} and \eqref{eq:num2} leads to the existence of a constant $c>0$ such that
\begin{equation}
\label{eq:upper}
\psi(\alpha)\le -c\ \Big( \frac{1}{\alpha}\wedge 1\Big).
\end{equation}
By \eqref{eq:prem}, the following upper-bound holds: for $f(x,y)=\log(y)$,
\begin{align}
\label{eq:bou}
Pf(t,h)-f(t,h)&\le \log\Big( 1+\frac{\rho_+}{r} \Big)-c\Big(\frac{1}{hr}\wedge 1\Big)\nonumber\\
&\le \frac{\rho_+}{r}-c\Big(\frac{1}{hr}\wedge 1\Big),\quad h\ge 0,\ t\ge 0.
\end{align}
Due to the definition of $\rho_+$, we know that 
\[
h\le \varphi(0)\vee(r+\rho_+)t\le \varphi(0)\vee(r+\rho_+)K,
\] 
where $\varphi$ is the boundary the process has to hit. 
In other words, there exist two constants $\kappa_0>0$ and $\kappa_1>0$ such that for any $\rho_+\le \kappa_0$ and any $(K,r)$ satisfying $(r+\kappa_0)K\le \kappa_1$ the following bound holds
\(
\rho_+\le \frac{c}{2}\,\Big(\frac{1}{h}\wedge r\Big).
\)
Hence:
\[
Pf(t,h)-f(t,h)\le -\frac{c}{2r}\Big(\frac{1}{\varphi(0)\wedge \kappa_1}\wedge r\Big)=:-\mathcal{R}^{-1}(r).
\]
We deduce that the function $g(t,h)$ defined by $g(t,h)=\mathcal{R}(r)\,(f(t,h)-\log\epsilon)$ satisfies $g(t,h)\ge 0$ for any $(t,h)\in E$ and $Pg(t,h)-g(t,h)\le -1$ on $E$. The potential theory therefore implies:
\[
\mathbb{E}[\mathcal{N}_{\epsilon,K}]\le g(0,\varphi(0))\le \mathcal{R}(r)(\log(\varphi(0))-\log(\epsilon)).
\] 
We finally deduce the existence of $a>0$ and $b>0$ such that $\mathbb{E}[\mathcal{N}_{\eps,K}]\le (a+br)|\log\epsilon|$ for $\epsilon$ small enough.\\
For the particular case of a non increasing boundary function it suffices to vanish $\rho_+$ in \eqref{eq:bou} and to apply the same arguments of the potential theory in order to get \eqref{eq:thm:conv13}
\end{proof}
\section{Examples and numerics.}\label{sec:examples}
In this section, we present three different examples which nicely illustrate the efficiency of these new algorithms \ref{algo1} and \ref{algo2}. 
\mathversion{bold}
\subsection{Brownian hitting time of $\varphi(t)=\sqrt{1+\alpha t}$ }
\mathversion{normal}
Let us first consider an application of Theorem \ref{thm1}. We observe that $\varphi(t)=\sqrt{1+\alpha t}$ is an increasing function satisfying \eqref{hypo}, \eqref{hypo+} and \eqref{hypo++} for $\alpha\in [0,1]$. Consequently  Algorithm \ref{algo1} converges and permits us to obtain an approximation of the hitting time $\tau_\varphi$. In the figures, we present the link between the averaged number of steps and $\epsilon$ which characterizes the approximation error size.\\
The first figure (resp. the second one) concerns: $\alpha=1$ (resp. $\alpha=0.01$), $\epsilon=0.5^n$ ($n$ is represented on the horizontal axis) and the number of simulation in order to estimate the averaged number of steps is $10\,000$. 

\begin{figure}[ht]
\begin{center}
\subfloat[{$\alpha = 1$}]{
\includegraphics[width=0.48\textwidth]{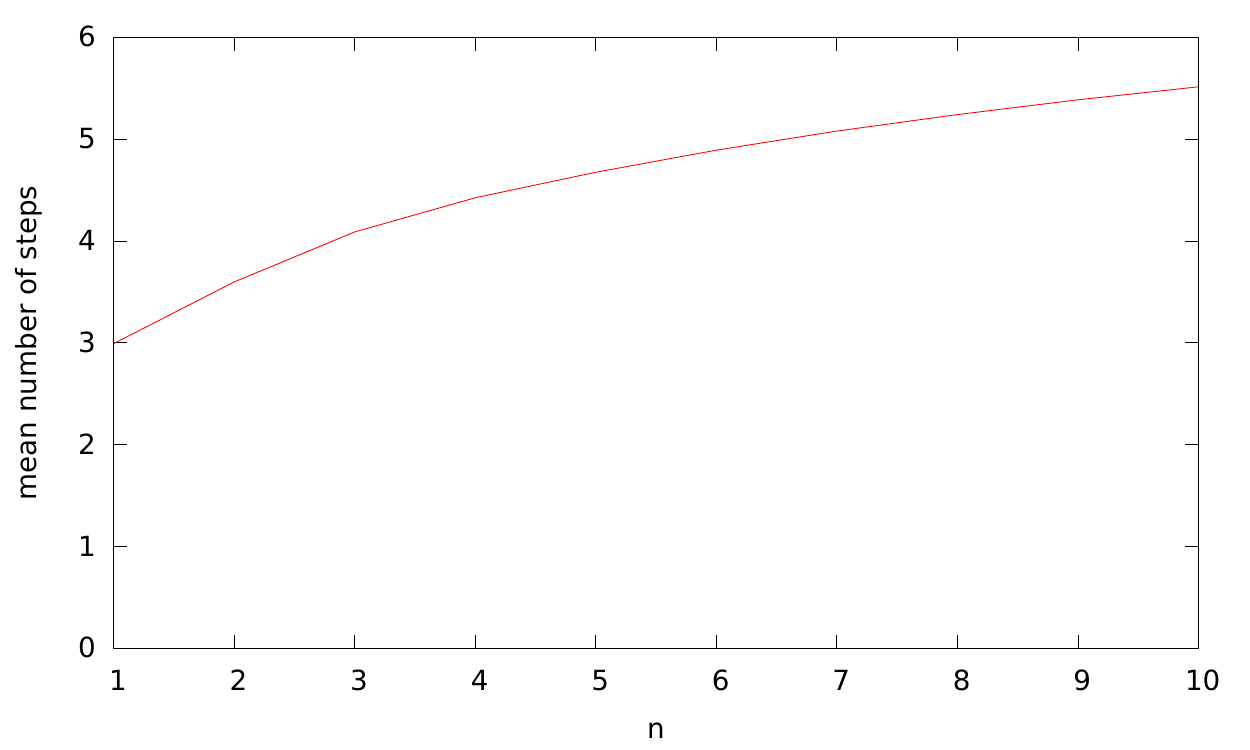}
}
\subfloat[{$\alpha = 0.01$}]{
\includegraphics[width=0.48\textwidth]{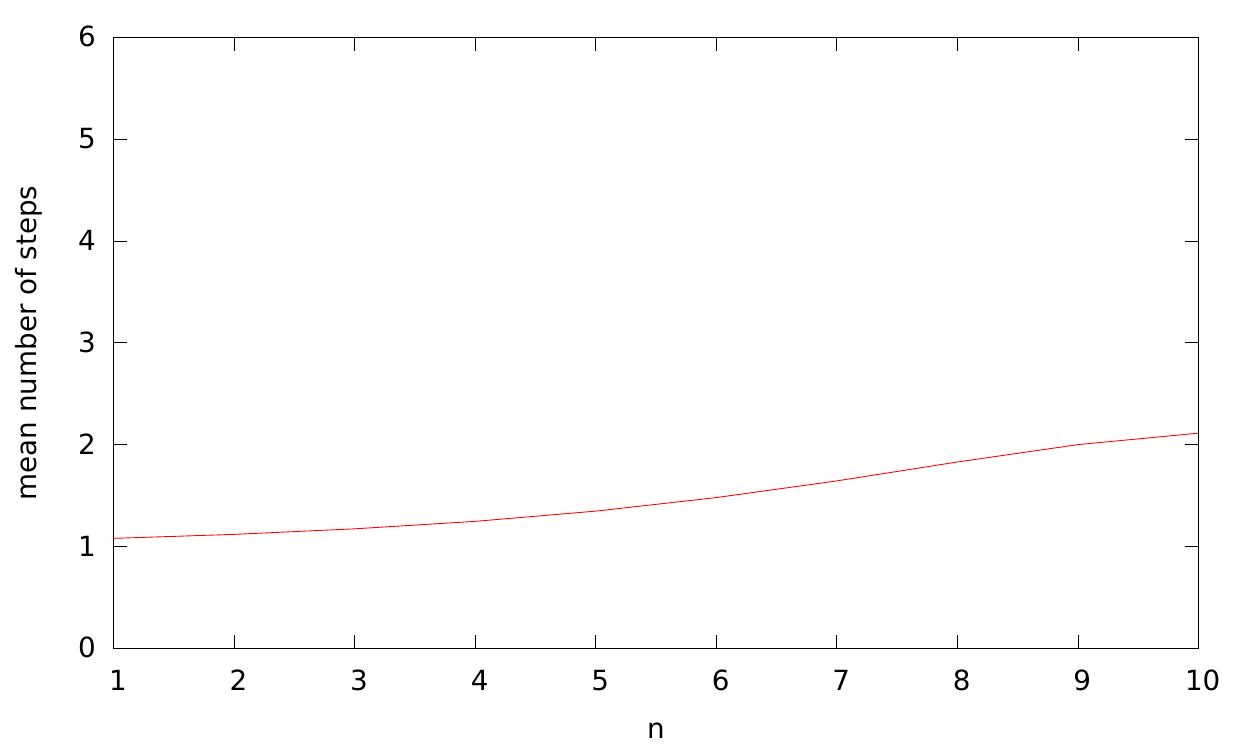}
}
\end{center} 
\caption{\(\mathbb{E}(\mathcal{N}_\epsilon)\): mean number of steps for \(\epsilon = 0.5^n\) as a function of \(n\). The boundary is 
$\varphi(t)=\sqrt{1+\alpha t}$.}\label{fig:nombredetapetmoyen}
\end{figure}
Let us now present the approximate distribution of the hitting time.
\begin{figure}[ht]
\begin{center}
\subfloat[{$\alpha = 1$}]{
\includegraphics[width=0.48\textwidth]{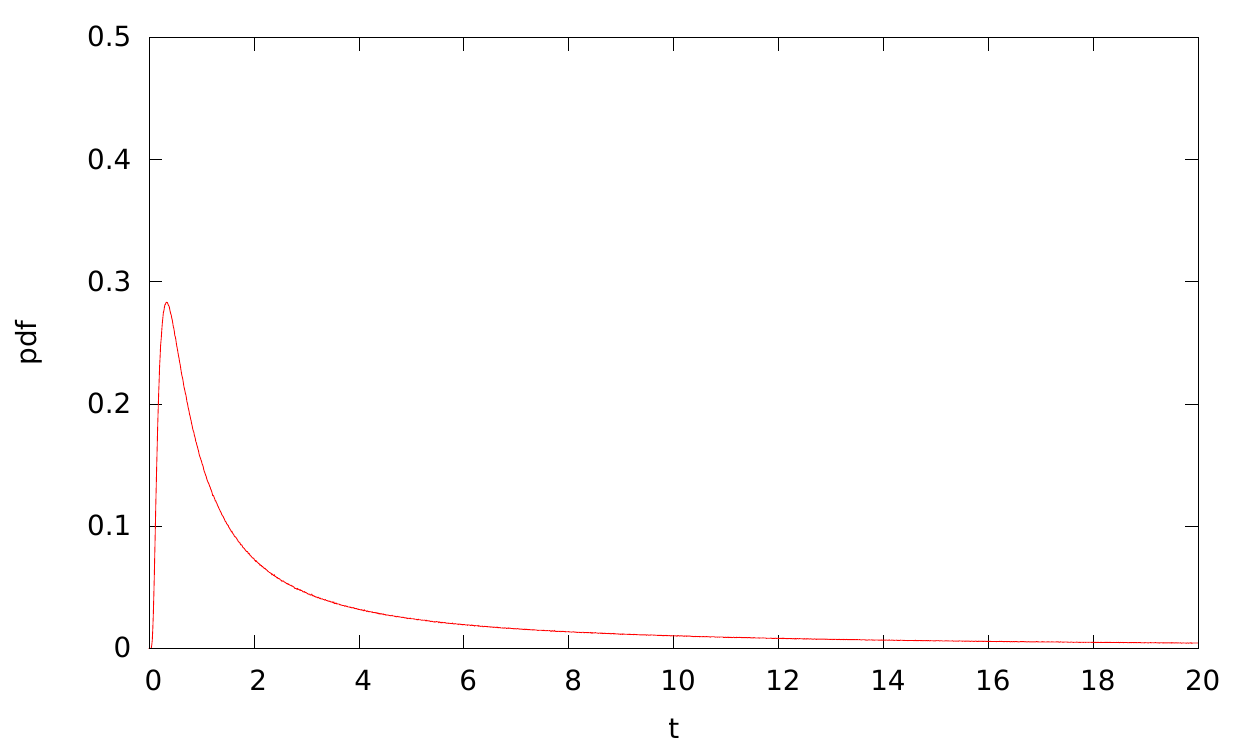}
}
\subfloat[{$\alpha = 0.01$}]{
\includegraphics[width=0.48\textwidth]{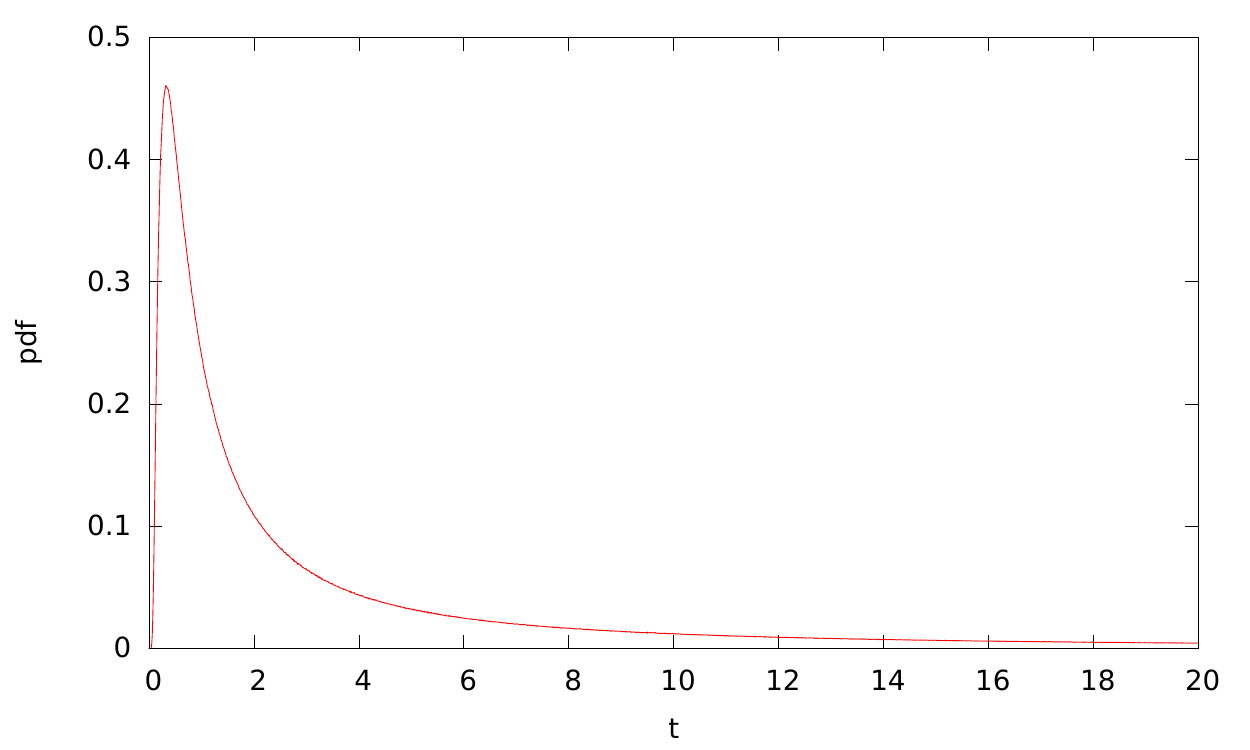}
}
\end{center} 
\caption{Empirical distribution of the approximate first hitting time of the boundary 
$\varphi(t)=\sqrt{1+\alpha t}$.}\label{fig:distribfht}
\end{figure}

\mathversion{bold}
\subsection{Brownian hitting time of $\varphi(t)=\alpha+\beta\cos(\omega t)$ }
\mathversion{normal}
Let us now consider the first time the Brownian motion hits the periodic boundary $\varphi(t)=\alpha+\beta\cos(\omega t)$. Since the boundary is not an increasing function, 
we shall use Algorithm \ref{algo2}.  Theorem \ref{thm2} ensures that the algorithm converges. Let us therefore use the Monte-Carlo method in order to estimate precisely 
the average number of steps. 
As explained in the previous section, the simulation procedure permits to approximation of the stopping time $\tau_\varphi\wedge K$ for some given fixed time $K$.
Figure~\ref{fig:mean_number_cos} illustrates the approximation \(\tau_\varphi\) by \(\tau^{\epsilon,K}_\varphi\), where the parameters are 
fixed at $\alpha=3.5$, $\beta=3$ and $\omega=\pi/2$. 
The maximal time are $K=20$ on one hand and \(K=100\) on the other hand and the error rate is given by $\epsilon=0.5^n$, for \(1\leq n\leq 10\). 
A sample of $10E8$ paths has been simulated to approximate the mean. 
\begin{figure}[h]
\begin{center}
\subfloat[\(\mathbb{E}(\mathcal{N}_{1/2^n,K})\) versus \(n\)]{
\includegraphics[width=0.48\textwidth]{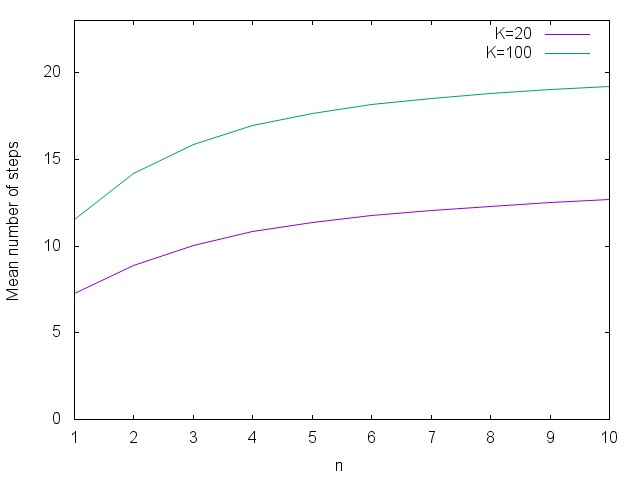}}
\subfloat[Distribution of \(\tau^{1/2^{n},K}_{\varphi}\). \(n=10\), \(K=20\).]{
\includegraphics[width=0.48\textwidth]{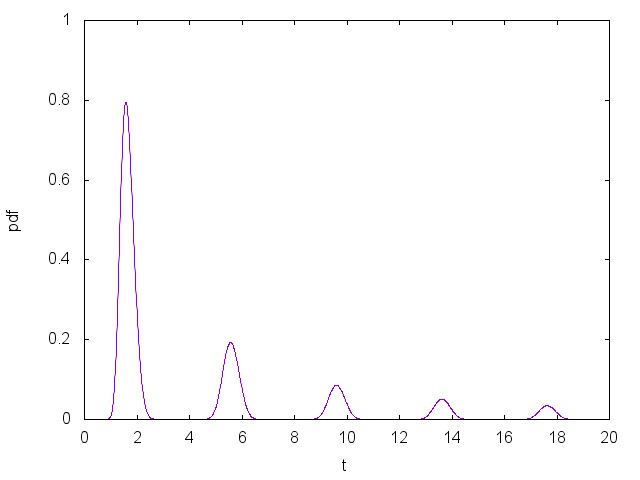}}
\end{center}
\caption{Approximation of \(\tau_{\varphi}\) with $\varphi(t)=3.5+3\cos(\pi t/2)$ }\label{fig:mean_number_cos}
\end{figure}
We know that the mean number of steps is a decreasing function of \(\epsilon\) and an increasing function of  \(K\). 
Figure~\ref{fig:effetoftruncation} gives the evolution of the mean number of steps as a function of the truncation \(K\).
\begin{figure}[h]
\begin{center}
\includegraphics[width=0.6\textwidth]{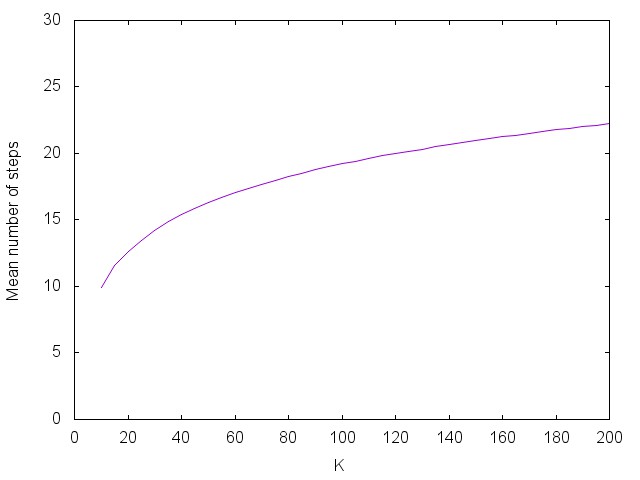}
\caption{\(\mathbb{E}(\mathcal{N}^{\epsilon,K}_\varphi)\) as a function of \(K\).}\label{fig:effetoftruncation}
\end{center}
\end{figure}
In practice, we obtained easily an impressively accurate approximation of \(\tau_\varphi\).

\subsection{The first time the Ornstein Uhlenbeck process hits the boundary $\varphi(t)=\alpha+\beta\cos(\omega t)$ }
\mathversion{normal}
This last section concerns a particular framework where passage times play a crucial role: the spiking neuron analysis. Let us roughly explain how neuronal firing activities has been modeled. The potential difference that exists across the cell membrane is modeled as an Ornstein-Uhlenbeck process. As soon as this membrane potential exceeds a given threshold, the neuron releases a rapid electrical signal called a spike and the membrane potential is directly reset to an initial voltage. Hence the interspike interval is identify with the first hitting time of an OU process whereas the spike train forms a renewal process. Such a stochastic leaky integrate and fire (LIF) neuronal model is a good compromise between realism and mathematical tractability \cite{Burkitt-2006, Gerstner, Tuckwell-1989, Tuckwell-Wan-2000}.

The standard OU process can be adapted when the inputs are time-dependent. It especially concerns many situations where the sensory stimuli, like sound, contain an oscillatory component. We observe then   oscillating membrane potentials in the neuron, generating rhythmic spiking patterns (see Iolov, Ditlevsen and Longtin \cite{Iolov-Ditlevsen-Longtin-2014} and references therein).

In such a model with time-dependent forcing, the membrane voltage denoted by $(V_t,\,t\ge 0)$ satisfies the following stochastic differential equation:
\[
dV_t=\Big( \chi(t)-\frac{V_t}{\tau}\Big)\, dt+\sigma\, dB_t,
\]
until it reaches the voltage threshold $V_{\rm th}$. Here $\chi$ represents a current acting on the cell, $\tau$ is the membrane time constant,  $\sigma$ is
the strength of the stochastic fluctuations and finally $(B_t,\, t\ge 0)$ stands for the standard Brownian motion. The length of the interspike interval is therefore directly related to the passage time of the stochastic process $(V_t)$ through the threshold $V_{\rm th}$.

Let us introduce a simple change of variable, given by $X_t=V_t-\varphi(t)$ with $\varphi$ the deterministic function satisfying:
\[
\varphi'(t)=\chi(t)-\frac{\varphi(t)}{\tau},\quad \varphi(0)=V_0.
\]
By straightforward computations, we can prove that $X_t$ is a classical OU process. In other words, the first passage time of the voltage $V_t$ through the given threshold $V_{\rm th}$ is almost surely equal to the first passage time of the OU process $(X_t)$ through the curved boundary $\varphi$. Simulating hitting times to curved boundaries for OU processes is therefore a main task.

That's why, we focus our attention to the last example which concerns the one-dimensional Ornstein-Uhlenbeck process defined  by:
\begin{equation}\label{eq:orn}
dX_t=dB_t-\lambda X_t\,dt,\quad X_0=x_0.
\end{equation}
The aim is to approximate the first passage time through the particular simple curved boundary $\varphi(t)=\alpha+\beta\cos(\omega t)$ where $\varphi(0)>x_0$.

Since the Ornstein-Uhlenbeck process can be represented as a time-changed Brownian motion, the question is directly related to the main results of this study. Indeed the solution of \eqref{eq:orn} is given by
\[
X_t=e^{-\lambda t}\left( x_0+\int_0^t e^{\lambda s}dB_s \right),\quad t\ge 0.
\]
Using Levy's theorem, $(X_t,\, t\ge 0)$ has the same distribution as $(Y_t,\, t\ge 0)$ defined by
\[
Y_t:=e^{-\lambda t}\left(x_0+W_{u(t)} \right),\quad t\ge 0,
\]
with $u(t):=\frac{1}{2\lambda}\,(e^{2\lambda t}-1)$ and $W$ a standard Brownian motion. We deduce that 
\[
\mathcal{T}_\varphi:=\inf\{t\ge 0:\ X_t=\varphi(t)\}
\]
has the same distribution as 
\begin{align*}
\hat{\mathcal{T}}_\varphi & :=\inf\Big\{ t\ge 0:\ e^{-\lambda t}\Big(x_0+W_{u(t)}\Big)=\varphi(t) \Big\}\\
&=\inf\Big\{u^{-1}(s)\ge 0:\ W_s=\varphi(u^{-1}(s)) e^{\lambda u^{-1}(s)}-x_0\Big\}\\
&=u^{-1}(\tau_\psi),
\end{align*}
where 
\[
\tau_\psi:=\inf\{t\ge 0:\, W_t=\psi(t)\},\quad \psi(t):=\sqrt{1+2\lambda t}\ \varphi\Big(\frac{\log(1+2\lambda t)}{2\lambda}\Big)-x_0.
\]
Consequently, in order to simulate the Ornstein-Uhlenbeck hitting time $\mathcal{T}_\varphi\wedge K$ for some $K$, we simply use Algorithm \ref{algo2} and propose an approximation of the Brownian hitting time $\tau_\psi\wedge \tilde{K}$ with $\tilde{K}:=u(K)=(e^{2\lambda K}-1)/(2\lambda)$.\\
Let us note that a straightforward computation leads to the following upper-bound:
\[
|\psi'(t)|\le \frac{\lambda\alpha+\lambda\beta+\omega\beta}{\sqrt{1+2\lambda t}}\le \lambda\alpha+\lambda\beta+\omega\beta,\quad t\ge 0.
\]
In other words, the continuous curve $\psi$ satisfies Hypothesis \eqref{hypo-new}: Algorithm \ref{algo2} therefore converges and Theorem \ref{thm2} can be applied.\\
In the following numerical experiences, we will choose $r=0.5+\lambda\alpha+\lambda\beta+\omega\beta$. 
Figures~\ref{fig:premiereou} and \ref{fig:oudeuxieme} concern the following choice of parameters: $x_0=0$, $\alpha=2$, $\beta=1$, $\omega=\pi/5$, $\lambda=0.5$.
We have chosen $K=5$ for Figure~\ref{fig:premiereou} and $K=10$ for Figure~\ref{fig:oudeuxieme}.
In both cases, the first figure represents the average number of steps as a function of \(n\) where the approximation parameter \(\epsilon\) is chosen as
\(0.5^n\), for \(n=1,\cdots,10\). 
The average has been estimated using $5.10E6$ simulations. 
The second figure represents the distribution  of $\mathcal{T}_\varphi\wedge K$ for \(n=10\). 

We observe that the \textit{change of time} \(\tilde{K} = (e^{2\lambda K}-1)/(2\lambda)\)  increases very fast with \(K\) and the number becomes quite large when \(K\) increases.
Note however that the number of random variables we have to simulate keeps relatively small in comparaison with the use of
a classical stopped Euler scheme usually used to approximate \(\mathcal{T}_\varphi\). 
Each Euler scheme introduces a bias, which goes to \(0\) as the time discretization length goes to \(0\). We have plotted the error on
Figure~\ref{fig:compare_euler_schemes} for the approximation of \(\mathbb{E}(\mathcal{T}_\varphi^K)\). We can observe that a simple 
Euler scheme has an order of convergence \(1/2\). We illustrate the improvement of this rate of convergence using two
particular modifications of the scheme: the first one from  \cite{Gobet-2000} (i.e. we take into account the first order term of 
the probability to hit the boundary between two time successive time-steps), the second-one from \cite{Gobet-Menozzi-10} (i.e. an adapted modification of the boundary). 
Both modifications yield to a scheme of order \(1\). 

The numerical cost of our algorithm increases very slowly as the parameter \(\epsilon\) goes to \(0\). 
The numerical comparisons are done with \(\epsilon=2^{-20}\), such that the error is almost negligible. 
The time we need is similar to the time for an Euler scheme with step \(0.01\) and the Brownian bridge modification with time step \(0.02\).
Empirically, we conclude that our scheme over performs previous ones if one needs an accuracy larger than those obtained with
an Euler scheme with time step \(0.01\).

\begin{figure}[h]
\begin{center}
\subfloat[\(\mathbb{E}(\mathcal{N}_{1/2^n,K})\) versus \(n\)]{
\includegraphics[width=0.48\textwidth]{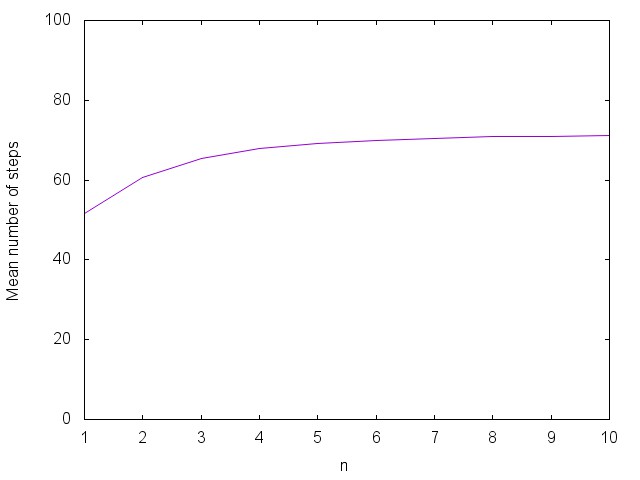}}
\subfloat[Distribution of \(\mathcal{T}^{K,\epsilon}_\varphi\)  (\(\epsilon = 1/2^{10}\)).]{
\includegraphics[width=0.48\textwidth]{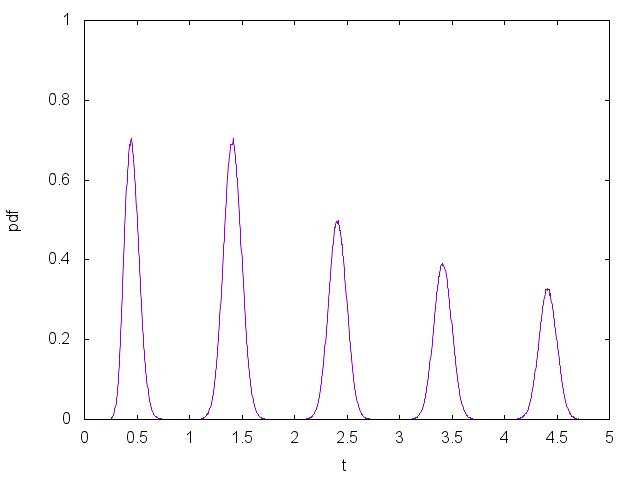}}
\end{center}
\caption{First hitting time of \(\varphi(t)=\alpha+\beta\cos(\omega t)\) by an Ornstein Uhlenbeck process solution of \eqref{eq:orn} ($\alpha=2$, $\beta=1$,  $\omega=2\pi$, $\lambda=0.5$, \(K=5\).)}\label{fig:premiereou}
\end{figure}

\begin{figure}[h]
\begin{center}
\subfloat[\(\mathbb{E}(\mathcal{N}_{1/2^n,K})\) versus \(n\)]{
\includegraphics[width=0.48\textwidth]{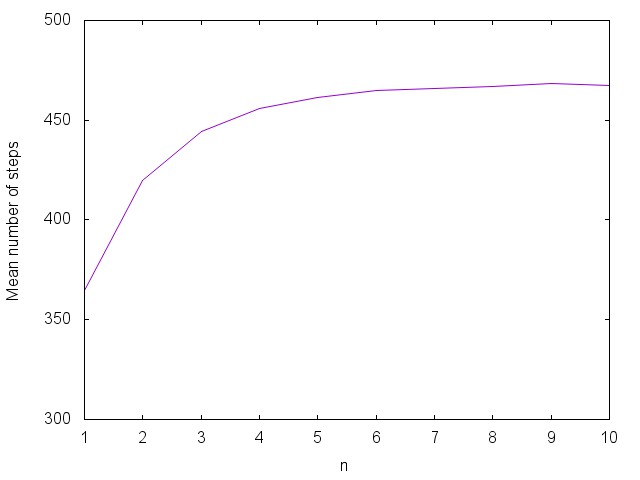}}
\subfloat[Distribution of \(\mathcal{T}^{K,\epsilon}_\varphi\)  (\(\epsilon = 1/2^{10}\)).]{
\includegraphics[width=0.48\textwidth]{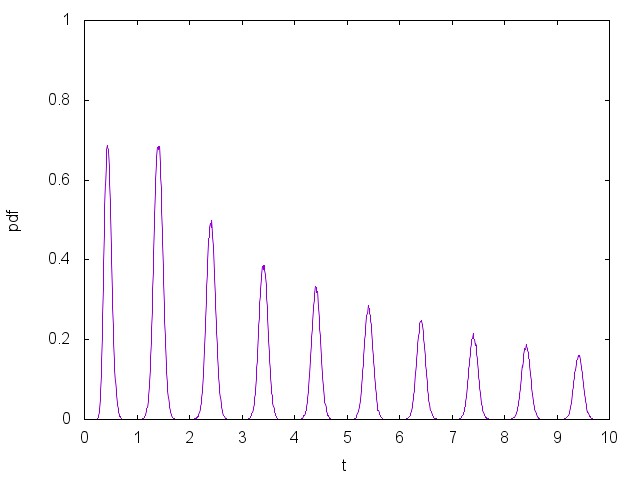}}
\end{center}
\caption{First hitting time of \(\varphi(t)=\alpha+\beta\cos(\omega t)\) by an Ornstein Uhlenbeck process solution of \eqref{eq:orn} ($\alpha=2$, $\beta=1$, $\omega=2\pi$, $\lambda=0.5$, \(K=10\)).}\label{fig:oudeuxieme}
\end{figure}

\begin{figure}[h]
\begin{center}
\includegraphics[width=0.48\textwidth]{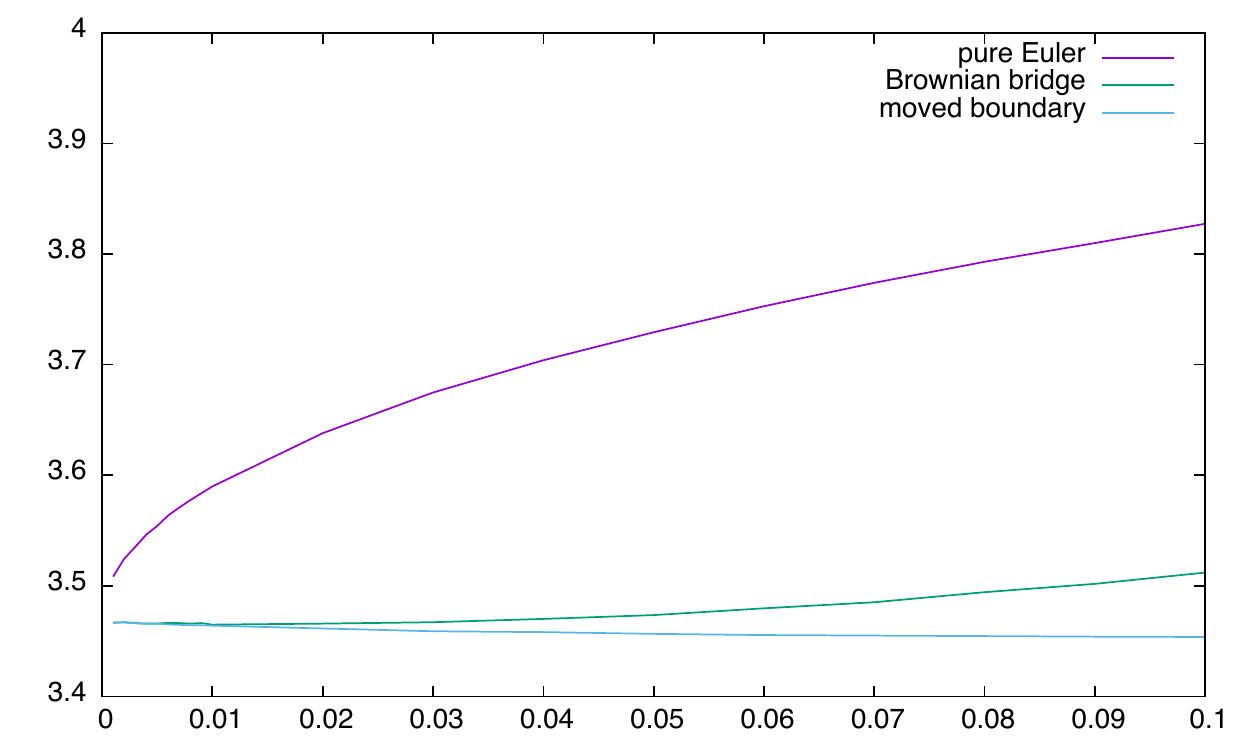}
\end{center}
\caption{Bias for Euler schemes to approximate \(\mathbb{E}(\mathcal{T}_\varphi^K)\), the mean first hitting time of \(\varphi(t)=\alpha+\beta\cos(\omega t)\) by an Ornstein Uhlenbeck process solution of \eqref{eq:orn} ($\alpha=2$, $\beta=1$, $\omega=2\pi$, $\lambda=0.5$, \(K=5\)).}\label{fig:compare_euler_schemes}
\end{figure}

\clearpage

\bibliographystyle{abbrv}
\bibliography{biblio}

\end{document}